\documentclass[a4paper]{amsart}
\usepackage[T1]{fontenc}
\usepackage[british]{babel}
\usepackage{stmaryrd}
\DeclareMathSymbol{\rest}{\mathbin}{AMSa}{"16}
\newcommand{\U}{\mathop{\textstyle\bigcup}}
\newcommand{\bool}[1]{\mathsf{#1}}
\newcommand{\tow}[1]{\mathcal{#1}}
\newcommand{\qp}[1]{\left[ #1 \right]}
\newcommand{\Qp}[1]{\left\llbracket #1 \right\rrbracket}
\newcommand{\ap}[1]{\left\langle #1 \right\rangle}
\newcommand{\bp}[1]{\left\lbrace #1 \right\rbrace}
\newcommand{\vp}[1]{\left\lvert #1 \right\rvert}
\DeclareMathOperator{\dom}{dom}
\DeclareMathOperator{\crit}{crit}
\DeclareMathOperator{\Ult}{Ult}
\DeclareMathOperator{\trcl}{trcl}
\DeclareMathOperator{\Coll}{\bool{Coll}}
\DeclareMathOperator{\RO}{\bool{RO}}
\DeclareMathOperator{\dens}{\pi}
\DeclareMathOperator{\sat}{sat}
\DeclareMathOperator{\Th}{Th}
\theoremstyle{plain}
\newtheorem{theorem}{Theorem}[section]
\newtheorem{proposition}[theorem]{Proposition}
\newtheorem{lemma}[theorem]{Lemma}
\newtheorem{corollary}[theorem]{Corollary}
\theoremstyle{definition}
\newtheorem{definition}[theorem]{Definition}
\theoremstyle{remark}
\newtheorem{remark}[theorem]{Remark}
\begin{document}
\title{Universality properties of forcing}
\author{Francesco Parente}
\address{Graduate School of System Informatics\\Kobe University\\1-1 Rokkodai-cho Nada-ku\\Kobe 657-8501\\Japan}
\author{Matteo Viale}
\address{Dipartimento di Matematica ``Giuseppe Peano''\\Universit\`a di Torino\\Via Carlo Alberto 10\\10123 Turin\\Italy}
\thanks{The first author is an International Research Fellow of the Japan Society for the Promotion of Science. The second author acknowledges support from INdAM through GNSAGA and from the projects: PRIN 2017 ``Mathematical Logic: models, sets, computability'', prot.\ 2017NWTM8R, and PRIN 2022 ``Models, sets and classifications'', prot.\ 2022TECZJA}
\begin{abstract}
The purpose of this paper is to investigate forcing as a tool to construct universal models. In particular, we look at theories of initial segments of the universe and show that any model of a sufficiently rich fragment of those theories can be embedded into a model constructed by forcing. Our results rely on the model-theoretic properties of good ultrafilters, for which we provide a new existence proof on non-necessarily complete Boolean algebras.
\end{abstract}
\maketitle

\section{Introduction}

There have been several attempts in the literature to account for the central role of forcing in modern set-theoretic practice. For instance, forcing axioms and generic absoluteness results are amenable to a model-theoretic analysis which singles out the peculiar model-theoretic properties of models of set theory where forcing axioms hold (see for example the second author's work on this matter \cite{viale:category_forcing,VIAAUD14,VIAASP,MR4713473,MR4749956,viale2022absolutemodelcompanionshipforcibility}). The present study complements and expands this view by highlighting a universality phenomenon arising from the Boolean-valued approach to forcing.

Section \ref{section:due} collects standard definitions and basic facts about Boolean algebras and Boolean-valued structures. To gauge the effect of forcing on initial segments of the universe, we introduce $\bool{B}$-valued structures of the form $H_{\check{\delta}}^\bool{B}$, which represent $H_\delta$ of the forcing extension by $\bool{B}$. Furthermore, we review normal and fine ideals and describe the Boolean algebra $\bool{B}(\tow{I})$ associated to a tower $\tow{I}$ of normal fine ideals.

The main technical ingredient of our analysis are the so-called good ultrafilters, to which Section \ref{section:tre} is dedicated. After a brief overview of the relevant model-theoretic concepts, we present a new proof, which may be of independent interest, of the existence of good ultrafilters on Boolean algebras which satisfy a disjointability condition. As opposed to standard constructions using independent families, our argument applies to a wide class of Boolean algebras, not necessarily complete.

In Section \ref{section:quattro}, we establish the first universality result for a fragment of the theory of $H_{\kappa^+}$, in a signature consisting of $\Delta_0$-definable relations. Namely, we fix an infinite regular cardinal $\kappa$ and, assuming the existence of an inaccessible cardinal $\delta>\kappa$, we use the collapsing algebra $\Coll(\kappa,{<}\delta)$ to prove that any model of the universal theory of $H_{\kappa^+}$ of cardinality at most $\delta$ can be embedded into a model of the form $H_{\check{\delta}}^{\Coll(\kappa,{<}\delta)}/U$. This approach, although technically simple, is not flexible enough to accommodate a richer signature, nor the case where $\kappa$ is singular.

We overcome these limitations and establish a more general universality result in Section \ref{section:cinque}. Assuming again the existence of an inaccessible $\delta$, for each infinite cardinal $\kappa<\delta$ we make use of the stationary tower $\bool{B}(\tow{I}^\kappa_{<\delta})$ of height $\delta$ and critical point $\kappa^+$ to obtain universal models for the full theory of $H_{\kappa^+}$, in a signature consisting of arbitrary relations. Moreover, under the additional assumption that $\delta$ is a Woodin cardinal, we show that the corresponding $\bool{B}(\tow{I}^\kappa_{<\delta})$-valued structure correctly represents $H_\delta$ in the forcing extension.

\section{Boolean-valued structures}\label{section:due}

We follow the standard notation and terminology of \cite{KOPPELBERG}. A \emph{Boolean algebra} is a complemented distributive lattice $\ap{\bool{B},\vee,\wedge,\neg,\bool{0},\bool{1}}$ such that $\bool{0}^\bool{B}\neq\bool{1}^\bool{B}$. Each Boolean algebra is equipped with a partial order, defined as $a\le b\iff a\vee b=b$. With respect to this ordering, for a subset $X\subseteq\bool{B}$ we denote
\[
\bigvee X=\sup(X)\quad\text{and}\quad\bigwedge X=\inf(X),
\]
whenever they actually exist.

We define
\[
\bool{B}^+=\bool{B}\setminus\bp{\bool{0}}
\]
and, for each $b\in\bool{B}^+$, we define the \emph{relative algebra}
\[
\bool{B}\rest b = \bp{a\in\bool{B} : a\le b},
\]
which is also a Boolean algebra with the natural operations.

\begin{definition} Let $\delta$ be a cardinal; a Boolean algebra $\bool{B}$ is \emph{${<}\delta$-c.c.}\ if every antichain in $\bool{B}$ has cardinality less than $\delta$. The \emph{saturation} of $\bool{B}$, in symbols $\sat(\bool{B})$, is the least cardinal $\delta$ such that $\bool{B}$ is ${<}\delta$-c.c.
\end{definition}

The following weakening of the ${<}\delta$-c.c., essentially due to Baumgartner and Taylor \cite{BAUTAY}, will play a role in Section \ref{section:cinque}.

\begin{definition} Let $\delta$ be a cardinal; a Boolean algebra $\bool{B}$ is \emph{${<}\delta$-presaturated} if for every collection of antichains $\bp{A_\alpha : \alpha<\lambda}$ with $\lambda<\delta$ and every $b\in\bool{B}^+$ there exists $d\in\bool{B}^+$ such that $d\le b$ and for all $\alpha<\lambda$
\[
\vp{\bp{a\in A_\alpha : a\wedge d>\bool{0}}}<\delta.
\]
\end{definition}

\begin{definition} Let $\bool{B}$ be a Boolean algebra; a subset $D\subseteq\bool{B}^+$ is \emph{dense} if for all $b\in\bool{B}^+$ there exists $d\in D$ such that $d\le b$. The \emph{density} of $\bool{B}$, in symbols $\dens(\bool{B})$, is the minimum cardinality of a dense subset of $\bool{B}$.
\end{definition}

\begin{definition} Let $\kappa$ be a cardinal and $\bool{B}$ a Boolean algebra; a subset $C\subseteq\bool{B}^+$ is \emph{${<}\kappa$-closed} if in $C$ all decreasing chains of length less than $\kappa$ have a lower bound.
\end{definition}

We shall focus our attention to special classes of Boolean algebras in which some suprema and infima are guaranteed to exist.

\begin{definition} Let $\delta$ be a cardinal; a Boolean algebra $\bool{B}$ is \emph{${<}\delta$-complete} if for all $X\in\qp{\bool{B}}^{<\delta}$ both $\bigvee X$ and $\bigwedge X$ exist. Furthermore, $\bool{B}$ is \emph{complete} if it is ${<}\delta$-complete for every $\delta$. 
\end{definition}

According to MacNeille \cite{macneille:completion}, for every partial order $P$ there exists a complete Boolean algebra $\bool{B}$ together with an order and incompatibility preserving map $i\colon P\to\bool{B}^+$ such that $i[P]$ is dense in $\bool{B}$. Furthermore, $\bool{B}$ is uniquely determined up to isomorphism and is called $\RO(P)$, the \emph{completion} of $P$.

Next, we introduce the notion of a Boolean-valued structure for a given signature. We shall restrict our attention to \emph{relational} signatures, that is to say, containing only relation symbols and constant symbols.

\begin{definition}[Rasiowa and Sikorski \cite{rs1953}]\label{definition:bvm} Let $L$ be a relational signature and $\bool{B}$ a Boolean algebra. A \emph{$\bool{B}$-valued structure} $\mathfrak{M}$ for $L$ consists of:
\begin{enumerate}
\item a non-empty domain $M$;
\item a function
\[
\begin{split}
M\times M&\longrightarrow\bool{B} \\
\ap{\tau,\sigma} &\longmapsto\Qp{\tau=\sigma}^\mathfrak{M}
\end{split}\ ;
\]
\item for each $n$-ary relation symbol $R\in L$, a function
\[
\begin{split}
^nM&\longrightarrow\bool{B} \\
\ap{\tau_1,\dots,\tau_n} &\longmapsto\Qp{R(\tau_1,\dots,\tau_n)}^\mathfrak{M}
\end{split}\ ;
\]
\item for each constant symbol $c\in L$, an element $c^\mathfrak{M}\in M$.
\end{enumerate}
We require that the following conditions hold:
\begin{itemize}
\item for all $\tau,\sigma,\pi\in M$
\begin{gather*}
\Qp{\tau=\tau}^\mathfrak{M}=\bool{1}, \\
\Qp{\tau=\sigma}^\mathfrak{M}=\Qp{\sigma=\tau}^\mathfrak{M}, \\
\Qp{\tau=\sigma}^\mathfrak{M}\wedge\Qp{\sigma=\pi}^\mathfrak{M}\le\Qp{\tau=\pi}^\mathfrak{M};
\end{gather*}
\item for every $n$-ary relation symbol $R\in L$ and $\tau_1,\dots,\tau_n,\sigma_1,\dots,\sigma_n\in M$,
\[
\Qp{R(\tau_1,\dots,\tau_n)}^\mathfrak{M}\wedge\bigwedge_{i=1}^n \Qp{\tau_i=\sigma_i}^\mathfrak{M}\le\Qp{R(\sigma_1,\dots,\sigma_n)}^\mathfrak{M}.
\]
\end{itemize}
\end{definition}

If $\varphi(x_1,\dots,x_n)$ is an $L$-formula and $\tau_1,\dots,\tau_n\in M$, the Boolean value $\Qp{\varphi(\tau_1,\dots,\tau_n)}^\mathfrak{M}$ is defined recursively as follows: first extend Definition \ref{definition:bvm} from atomic formulae with parameters from $M$ to quantifier-free formulae by letting
\[
\Qp{\lnot\varphi}^\mathfrak{M}=\neg\Qp{\varphi}^\mathfrak{M}\quad\text{and}\quad\Qp{\varphi\land\psi}^\mathfrak{M}=\Qp{\varphi}^\mathfrak{M}\wedge\Qp{\psi}^\mathfrak{M}.
\]
Then, given an $L$-formula $\varphi(x,y_1,\dots,y_n)$ and $\sigma_1,\dots,\sigma_n\in M$, let
\[
\Qp{\exists x\varphi(x,\sigma_1,\dots,\sigma_n)}^\mathfrak{M}=\bigvee_{\tau\in M}\Qp{\varphi(\tau,\sigma_1,\dots,\sigma_n)}^\mathfrak{M},
\]
where the supremum on the right-hand side is computed in $\RO(\bool{B})$. This extension is legitimate because the dense embedding $i\colon\bool{B}\to\RO(\bool{B})$ preserves all the suprema and infima which happen to exist in $\bool{B}$.

\begin{definition} A $\bool{B}$-valued structure $\mathfrak{M}$ for $L$ is \emph{full} if for every $L$-formula $\varphi(x,y_1,\dots,y_n)$ and $\sigma_1,\dots,\sigma_n\in M$ there exist $\tau_1,\dots,\tau_m\in M$ such that
\[
\Qp{\exists x\varphi(x,\sigma_1,\dots,\sigma_n)}^\mathfrak{M}=\bigvee_{i=1}^m\Qp{\varphi(\tau_i,\sigma_1,\dots,\sigma_n)}^\mathfrak{M}.
\]
\end{definition}

Closely related is the mixing property which, together with fullness, will guarantee that the quotient operation described in Section \ref{section:tre} is well behaved.

\begin{definition} A $\bool{B}$-valued structure $\mathfrak{M}$ satisfies the \emph{${<}\delta$-mixing property} if for every antichain $A$ in $\bool{B}$ with $\vp{A}<\delta$ and every $\bp{\tau_a: a\in A}\subseteq M$ there exists $\tau\in M$ such that for all $a\in A$, $a\le\Qp{\tau=\tau_a}^\mathfrak{M}$. Furthermore, $\mathfrak{M}$ satisfies the \emph{mixing property} if it satisfies the ${<}\delta$-mixing property for every cardinal $\delta$.
\end{definition}

\begin{proposition}[{Viale \cite[Proposition 6.3.14]{viale:method}}]\label{proposition:mixingfull} Let $\bool{B}$ be a complete Boolean algebra. If a $\bool{B}$-valued structure $\mathfrak{M}$ satisfies the ${<}\delta$-mixing property for some cardinal $\delta>\vp{M}$, then $\mathfrak{M}$ is full.
\end{proposition}

In this paper, we shall be mostly focused on Boolean-valued models of set theory. As argued by Venturi and Viale \cite[Section 3]{VIAVEN23}, the appropriate signature to axiomatize them is $\in_{\Delta_0}$, which is the expansion of the signature $\bp{\in}$ containing an $n$-ary relation symbol $R_\varphi$ for each $\Delta_0$ $\bp{\in}$-formula $\varphi(x_1,\dots,x_n)$, as well as constant symbols for $\emptyset$ and $\omega$.

For a cardinal $\lambda$, let $H_\lambda=\bp{x : \vp{\trcl(x)}<\lambda}$. Whenever $\lambda$ is uncountable, we consider $H_\lambda$ as an $\in_{\Delta_0}$-structure, where the interpretation of the symbol $R_\varphi$ is the set of tuples in $H_\lambda$ which satisfy $\varphi$. Furthermore, if $\mathcal{A}$ is any set of relations on $H_\lambda$, let us say $\mathcal{A}\subseteq\bigcup_{n<\omega}\mathcal{P}(^n(H_\lambda))$, then we let ${\in_\mathcal{A}}={\in_{\Delta_0}}\cup\mathcal{A}$ be the signature obtained naming each element of $\mathcal{A}$ by a new relation symbol.

From now on, we assume some familiarity with the forcing method. Given a notion of forcing $P$, let $V^P$ be the class of $P$-names. For $\tau\in V^P$, let $\tau_G$ be the valuation of $\tau$ by a $P$-generic filter $G$. According to the classic approach of Scott, Solovay, and Vop\v enka, if $\bool{B}$ is a complete Boolean algebra then $V^\bool{B}$ is a $\bool{B}$-valued structure for $\in_{\Delta_0}$ with the mixing property. Indeed if $\varphi(x_1,\dots,x_n)$ is a $\Delta_0$ $\bp{\in}$-formula, then the interpretation of $R_\varphi$ is naturally given by
\[
\Qp{R_\varphi(\tau_1,\dots,\tau_n)}^{V^\bool{B}}=\Qp{\varphi(\tau_1,\dots,\tau_n)}^{V^\bool{B}}.
\]
Furthermore, the constant symbols for $\emptyset$ and $\omega$ are interpreted as $\emptyset$ and $\check{\omega}$ respectively.

Let $\delta$ be a cardinal; we say that a complete Boolean algebra $\bool{B}$ is \emph{$\delta$-preserving} if $\Qp{\check{\delta}\text{ is a cardinal}}^{V^\bool{B}}=\bool{1}$. In this case, we have a canonical set of names for $H_\delta$ as computed in the forcing extension by $\bool{B}$, given by the following definition.

\begin{definition}[{Asper\'o and Viale \cite[Definition 2.2]{VIAASP}}] Let $\delta$ be a cardinal and $\bool{B}$ be a $\delta$-preserving complete Boolean algebra. We define
\[
H_{\check{\delta}}^\bool{B}=\bp{\tau\in V^\bool{B}\cap V_\alpha : \Qp{\vp{\trcl(\tau)}<\check{\delta}}^{V^\bool{B}}=\bool{1}},
\]
where $\alpha\ge\vp{\bool{B}}$ is a sufficiently large ordinal such that for all $\sigma\in V^\bool{B}$
\[
\Qp{\vp{\trcl(\sigma)}<\check{\delta}}^{V^\bool{B}}\le\bigvee_{\tau\in V^\bool{B}\cap V_\alpha}\Qp{\tau=\sigma}^{V^\bool{B}}.
\]
\end{definition}

\begin{proposition}\label{proposition:rep} If $\bool{B}$ is $\delta$-preserving and $G$ is $\bool{B}$-generic over $V$, then
\[
H_\delta^{V[G]}=\bp{\tau_G : \tau\in H_{\check{\delta}}^\bool{B}}.
\]
\end{proposition}
\begin{proof} If $\tau\in H_{\check{\delta}}^\bool{B}$, then obviously $\tau_G\in H_\delta^{V[G]}$. Conversely, suppose $x\in H_\delta^{V[G]}$ and let $\dot{x}$ be a $\bool{B}$-name for $x$. By the forcing theorem $\Qp{\vp{\trcl(\dot{x})}<\check{\delta}}^{V^\bool{B}}\in G$ and, since $\alpha$ is sufficiently large, we may also assume that $\dot{x}\in V^\bool{B}\cap V_\alpha$. Let $\sigma$ be an arbitrary element of $H_{\check{\delta}}^\bool{B}$; using the mixing property of $V^\bool{B}$, we can easily find $\tau\in V^\bool{B}\cap V_\alpha$ such that
\[
\Qp{\vp{\trcl(\dot{x})}<\check{\delta}}^{V^\bool{B}}\le\Qp{\tau=\dot{x}}^{V^\bool{B}}\quad\text{and}\quad\neg\Qp{\vp{\trcl(\dot{x})}<\check{\delta}}^{V^\bool{B}}\le\Qp{\tau=\sigma}^{V^\bool{B}}.
\]
Then it is clear that $\Qp{\vp{\trcl(\tau)}<\check{\delta}}^{V^\bool{B}}=\bool{1}$ and $\tau_G=\dot{x}_G=x$, as desired.
\end{proof}

The next lemma shows that, in some cases, using $\bool{B}$-names in $H_\delta$ yields the same outcome as Proposition \ref{proposition:rep}.

\begin{lemma}[{Audrito and Viale \cite[Lemma 5.1]{VIAAUD14}}]\label{lemma:av} Let $\delta$ be a regular cardinal and $\bool{B}\subseteq H_\delta$ be a ${<}\delta$-presaturated Boolean algebra. If $G$ is $\bool{B}$-generic over $V$, then $\delta$ is a regular cardinal in $V[G]$ and
\[
H_\delta^{V[G]}=\bp{\tau_G : \tau\in V^\bool{B}\cap H_\delta}.
\]
\end{lemma}

For the next proposition, we equip $H_{\check{\delta}}^\bool{B}$ with the Boolean values inherited from $V^\bool{B}$. A proof can be found in Venturi and Viale \cite{VIAVEN23}, in particular Lemma 4.10 and the discussion preceding it.

\begin{proposition}\label{proposition:mixing} Let $\delta$ be an uncountable cardinal and $\bool{B}$ be a $\delta$-preserving complete Boolean algebra. Then $H_{\check{\delta}}^\bool{B}$ is a $\bool{B}$-valued structure for $\in_{\Delta_0}$ with the mixing property.
\end{proposition}

\begin{remark}\label{remark:abs} Let $\delta$ be an uncountable cardinal and $\bool{B}$ be a $\delta$-preserving complete Boolean algebra. For every quantifier-free $\in_{\Delta_0}$-formula $\varphi(x,y_1,\dots,y_n)$ and parameters $a_1,\dots,a_n\in H_{\delta}$ we have
\[
\ap{H_\delta,\in_{\Delta_0}}\models\exists x\varphi(x,a_1,\dots,a_n)\implies\Qp{\exists x\varphi(x,\check{a}_1,\dots,\check{a}_n)}^{H_{\check{\delta}}^\bool{B}}=\bool{1}.
\]
\end{remark}

The above observation follows easily from the Boolean-valued version of absoluteness for $\Delta_0$ formulae. However, preservation of universal formulae is more delicate: here is a lemma which, according to Kunen \cite[Lemma 4]{kunen:saturated}, is ``part of the folklore of the subject''. For further details in the context of forcing axioms, we refer the reader to \cite[Lemma 1.3]{VIAMMREV}.

\begin{lemma}\label{lemma:coll} Let $\kappa$ be an infinite regular cardinal and $\bool{B}$ be a complete Boolean algebra having a ${<}\kappa$-closed dense subset. For every quantifier-free $\in_{\Delta_0}$-formula $\varphi(x,y_1,\dots,y_n)$ and parameters $a_1,\dots,a_n\in H_{\kappa^+}$ we have
\[
\ap{H_{\kappa^+},\in_{\Delta_0}}\models\forall x\varphi(x,a_1,\dots,a_n)\implies\Qp{\forall x\varphi(x,\check{a}_1,\dots,\check{a}_n)}^{V^\bool{B}}=\bool{1}.
\]
\end{lemma}

In the final part of this section, we consider Boolean-valued structures of the form $\Ult(H_{\kappa^+},\tow{I})$ arising from towers of ideals $\tow{I}$. Given a set $X$, recall that an ideal $I$ over $\mathcal{P}(X)$ is
\begin{itemize}
\item \emph{normal} if for all $\bp{a_x : x\in X}\subseteq I$ we have $\bp{Z\subseteq X : \exists x\in Z(Z\in a_x)}\in I$;
\item \emph{fine} if for all $x\in X$ we have $\bp{Z\subseteq X : x\notin Z}\in I$.
\end{itemize}

\begin{definition} Let $\delta$ be an inaccessible cardinal. We say that $\tow{I}$ is a \emph{tower} of height $\delta$ if $\tow{I}=\ap{I_X : X\in V_\delta}$ where
\begin{enumerate}
\item\label{tow1} for all $X\in V_\delta$, $I_X$ is a normal fine ideal over $\mathcal{P}(X)$;
\item\label{tow2} for all $X\subseteq Y\in V_\delta$, $I_X=\bp{a\subseteq\mathcal{P}(X) : \bp{Z\subseteq Y : Z\cap X\in a}\in I_Y}$. 
\end{enumerate}
\end{definition}

Following the framework of Foreman \cite[Section 4.8]{FOREMAN}, to each tower $\tow{I}=\ap{I_X : X\in V_\delta}$ we associate a Boolean algebra $\bool{B}(\tow{I})$ as follows: by condition \eqref{tow2}, whenever $X\subseteq Y\in V_\delta$ we have an embedding of Boolean algebras
\[
\begin{split}
\iota_{X,Y}\colon\mathcal{P}(\mathcal{P}(X))/I_X&\longrightarrow\mathcal{P}(\mathcal{P}(Y))/I_Y \\
\qp{a}_{I_X} &\longmapsto \qp{\bp{Z\subseteq Y : Z\cap X\in a}}_{I_Y}
\end{split}.
\]
We let $\bool{B}(\tow{I})$ be the direct limit of the directed system of embeddings $\bp{\iota_{X,Y} : X\subseteq Y\in V_\delta}$.
More explicitly, $\bool{B}(\tow{I})$ can be described as the union algebra $\bigcup\bp{\mathcal{P}(\mathcal{P}(X)) : X\in V_\delta}$ quotiented by the equivalence relation which identifies $a\subseteq\mathcal{P}(X)$ and $a'\subseteq\mathcal{P}(X')$ if and only if $\iota_{X,X\cup X'}\bigl(\qp{a}_{I_X}\bigr)=\iota_{X',X\cup X'}\bigl(\qp{a'}_{I_{X'}}\bigr)$. Modulo this identification, we simply denote elements of $\bool{B}(\tow{I})$ as equivalence classes $\qp{a}_\tow{I}$.

The Boolean-algebraic properties of $\bool{B}(\tow{I})$ are summarized in the following two lemmas.

\begin{lemma}[{Foreman \cite[Lemma 4.49]{FOREMAN}}]\label{lemma:tow} Let $\delta$ be an inaccessible cardinal. If $\tow{I}$ is a tower of height $\delta$, then the Boolean algebra $\bool{B}(\tow{I})$ has cardinality $\le\delta$ and thus may be identified with a subset of $H_\delta$. Furthermore, $\bool{B}(\tow{I})$ is ${<}\delta$-complete, where by normality the supremum of $\bp{\qp{a_\alpha}_\tow{I} : \alpha<\lambda}$ for $\lambda<\delta$ is computed as follows: if $X\in V_\delta$ is sufficiently large such that $\lambda\subseteq X$ and $\U a_\alpha\subseteq X$ for each $\alpha<\lambda$, then
\[
\bigvee\bp{\qp{a_\alpha}_\tow{I} : \alpha<\lambda}=\qp{\bp{Z\subseteq X : \exists\alpha\in Z\cap\lambda (Z\cap\U a_\alpha\in a_\alpha)}}_\tow{I}.
\]
\end{lemma}

\begin{lemma}[{Burke \cite[Lemma 2.4]{BUR97}}]\label{lemma:bur} If $\tow{I}$ is a tower of inaccessible height $\delta$, then every antichain in $\bool{B}(\tow{I})$ can be represented by a disjoint sequence, in the sense that for each antichain $\bp{\qp{a_\alpha}_\tow{I} : \alpha<\lambda}$ there exists a sequence $\ap{b_\alpha : \alpha<\lambda}$ such that for all $\alpha<\lambda$ $\qp{b_\alpha}_\tow{I}=\qp{a_\alpha}_\tow{I}$ and whenever $\beta<\alpha$
\[
\bp{Z\subseteq\U b_\alpha\cup\U b_\beta : Z\cap\U b_\alpha\in b_\alpha\text{ and }Z\cap\U b_\beta\in b_\beta}=\emptyset.
\]
\end{lemma}

Having introduced the Boolean algebra $\bool{B}(\tow{I})$, we now discuss $\bool{B}(\tow{I})$-valued structures arising from the tower $\tow{I}$.

\begin{definition} Let $\tow{I}$ be a tower of inaccessible height $\delta$. For every infinite cardinal $\kappa<\delta$ we define
\[
\Ult(H_{\kappa^+},\tow{I})=\bp{f\colon\mathcal{P}(X)\to H_{\kappa^+} : X\in V_\delta}.
\]
For each $a\in H_{\kappa^+}$, let $c_a$ be the constant function
\[
\begin{split}
c_a\colon\mathcal{P}(\emptyset)&\longrightarrow H_{\kappa^+} \\
\emptyset &\longmapsto a
\end{split}
\]
in $\Ult(H_{\kappa^+},\tow{I})$.

If $\mathcal{A}\subseteq\bigcup_{n<\omega}\mathcal{P}(^n(H_{\kappa^+}))$ is any set of relations, we equip $\Ult(H_{\kappa^+},\tow{I})$ with a $\bool{B}(\tow{I})$-valued structure for $\in_\mathcal{A}$ as follows: given an $n$-ary relation $R$, whether it be the equality relation or any relation in $\mathcal{A}$, and given $f_i\colon\mathcal{P}(X_i)\to H_{\kappa^+}$ for $1\le i\le n$, define
\[
\Qp{R(f_1,\dots,f_n)}^{\Ult(H_{\kappa^+},\tow{I})}=\qp{\bp{Z\subseteq X_1\cup\dots\cup X_n : R(f_1(Z\cap X_1),\dots,f_n(Z\cap X_n))}}_\tow{I}.
\]
Furthermore, the constant symbols for $\emptyset$ and $\omega$ are interpreted as $c_\emptyset$ and $c_\omega$, respectively.
\end{definition}

\begin{remark} If $\mathcal{C}\subseteq V_\delta$ is cofinal, in the sense that for all $X\in V_\delta$ there exists $C\in\mathcal{C}$ such that $X\subseteq C$, then both $\bool{B}(\tow{I})$ and $\Ult(H_{\kappa^+},\tow{I})$ are completely determined by $\ap{I_X : X\in\mathcal{C}}$. This observation will be useful in Section \ref{section:cinque}, where we shall consider towers of the form $\ap{I_X : \kappa^+\subseteq X\in V_\delta}$.
\end{remark}

\begin{proposition}\label{proposition:full} Suppose $\tow{I}$ is a tower of inaccessible height $\delta$; let also $\kappa<\delta$ be an infinite cardinal and $\mathcal{A}\subseteq\bigcup_{n<\omega}\mathcal{P}(^n(H_{\kappa^+}))$. Then $\Ult(H_{\kappa^+},\tow{I})$ is a full $\bool{B}(\tow{I})$-valued structure for $\in_\mathcal{A}$ with the ${<}\delta$-mixing property. Moreover, given an $\in_\mathcal{A}$-formula $\varphi(x_1,\dots,x_n)$ and functions $f_i\colon\mathcal{P}(X_i)\to H_{\kappa^+}$ for $1\le i\le n$, letting $X=X_1\cup\dots\cup X_n$ yields
\begin{equation}\label{eq:value}
\Qp{\varphi(f_1,\dots,f_n)}^{\Ult(H_{\kappa^+},\tow{I})}=\qp{\bp{Z\subseteq X : \ap{H_{\kappa^+},\in_\mathcal{A}}\models\varphi(f_1(Z\cap X_1),\dots,f_n(Z\cap X_n))}}_\tow{I}.
\end{equation}
\end{proposition}
\begin{proof} First of all, since the conditions of Definition \ref{definition:bvm} are easily verified, $\Ult(H_{\kappa^+},\tow{I})$ is indeed a $\bool{B}(\tow{I})$-valued structure for $\in_\mathcal{A}$.

To establish the ${<}\delta$-mixing property, let $\lambda<\delta$ and suppose we are given an antichain $\bp{\qp{a_\alpha}_\tow{I} : \alpha<\lambda}$ in $\bool{B}(\tow{I})$ as well as functions $f_\alpha\colon\mathcal{P}(X_\alpha)\to H_{\kappa^+}$ for $\alpha<\lambda$. Choose $Y\in V_\delta$ sufficiently large that $X_\alpha\cup\U a_\alpha\subseteq Y$ for each $\alpha<\lambda$. By Lemma~\ref{lemma:bur}, we may assume without loss of generality that for all $\beta<\alpha$
\begin{equation}\label{eq:unique}
\bp{Z\subseteq Y : Z\cap\U a_\alpha\in a_\alpha\text{ and }Z\cap\U a_\beta\in a_\beta}=\emptyset.
\end{equation}
Now let $f\colon\mathcal{P}(Y)\to H_{\kappa^+}$ be defined as follows: given $Z\subseteq Y$, we distinguish two cases. If there exists $\alpha<\lambda$ such that $Z\cap\U a_\alpha\in a_\alpha$, then we let
\[
f(Z)=f_\alpha(Z\cap X_\alpha);
\]
note that, in this case, $\alpha$ must be unique by \eqref{eq:unique}. Otherwise, let $f(Z)$ be defined arbitrarily. Then it is clear that $\qp{a_\alpha}_\tow{I}\le\Qp{f=f_\alpha}^{\Ult(H_{\kappa^+},\tow{I})}$ for each $\alpha<\lambda$, as desired.

Finally, we establish by induction fullness and \eqref{eq:value} at the same time. For quantifier-free formulae there is nothing to prove, so assume $\varphi(x,y_1,\dots,y_n)$ satisfies \eqref{eq:value} for any choice of parameters and let $f_i\colon\mathcal{P}(X_i)\to H_{\kappa^+}$ for $1\le i\le n$ be given. Letting $X=X_1\cup\dots\cup X_n$ and temporarily dropping the superscript $\Ult(H_{\kappa^+},\tow{I})$ from the Boolean values, the inductive hypothesis implies that for all $g\in\Ult(H_{\kappa^+},\tow{I})$
\[
\Qp{\varphi(g,f_1,\dots,f_n)}\le\qp{\bp{Z\subseteq X : \ap{H_{\kappa^+},\in_\mathcal{A}}\models\exists x\varphi(x,f_1(Z\cap X_1),\dots,f_n(Z\cap X_n))}}_\tow{I}
\]
and therefore
\begin{equation}\label{eq:ma1}
\Qp{\exists x\varphi(x,f_1,\dots,f_n)}\le\qp{\bp{Z\subseteq X : \ap{H_{\kappa^+},\in_\mathcal{A}}\models\exists x\varphi(x,f_1(Z\cap X_1),\dots,f_n(Z\cap X_n))}}_\tow{I}.
\end{equation}
For the other inequality, we define a function $f\colon\mathcal{P}(X)\to H_{\kappa^+}$ as follows: given $Z\subseteq X$, if
\[
\ap{H_{\kappa^+},\in_\mathcal{A}}\models\exists x\varphi(x,f_1(Z\cap X_1)\dots,f_n(Z\cap X_n))
\]
then choose $f(Z)\in H_{\kappa^+}$ such that
\[
\ap{H_{\kappa^+},\in_\mathcal{A}}\models\varphi(f(Z),f_1(Z\cap X_1),\dots,f_n(Z\cap X_n));
\]
otherwise, let $f(Z)$ be arbitrary. From the definition of $f$ and the inductive hypothesis, it follows that
\begin{equation}\label{eq:ma2}
\qp{\bp{Z\subseteq X : \ap{H_{\kappa^+},\in_\mathcal{A}}\models\exists x\varphi(x,f_1(Z\cap X_1),\dots,f_n(Z\cap X_n))}}_\tow{I}\le\Qp{\varphi(f,f_1,\dots,f_n)}.
\end{equation}
Combining \eqref{eq:ma1} and \eqref{eq:ma2}, we obtain equality throughout.
\end{proof}

Equation \eqref{eq:value} in Proposition \ref{proposition:full} leads directly to the following corollary.

\begin{corollary}\label{corollary:abs} Let $\tow{I}$ be a tower of inaccessible height $\delta$, let $\kappa<\delta$ be an infinite cardinal and $\mathcal{A}\subseteq\bigcup_{n<\omega}\mathcal{P}(^n(H_{\kappa^+}))$. For every $\in_\mathcal{A}$-formula $\varphi(x_1,\dots,x_n)$ and parameters $a_1,\dots,a_n\in H_{\kappa^+}$ we have
\[
\ap{H_{\kappa^+},\in_\mathcal{A}}\models\varphi(a_1,\dots,a_n)\iff\Qp{\varphi(c_{a_1},\dots,c_{a_n})}^{\Ult(H_{\kappa^+},\tow{I})}=\bool{1}.
\]
\end{corollary}

\section{Good ultrafilters and saturation}\label{section:tre}

Boolean-valued structures can be quotiented by an ultrafilter to recover a standard Tarski structure. This section is focused on how the combinatorial properties of the ultrafilter affect the model-theoretic properties of the quotient.

\begin{definition} Let $\bool{B}$ be a Boolean algebra and $\mathfrak{M}$ a $\bool{B}$-valued structure for $L$. Given an ultrafilter $U$ on $\bool{B}$, the \emph{quotient} of $\mathfrak{M}$ by $U$ is the $L$-structure $\mathfrak{M}/U$ defined as follows:
\begin{itemize}
\item its domain is $M/U=\bp{\qp{\tau}_U : \tau\in M}$, where $\qp{\tau}_U=\bp{\sigma\in M : \Qp{\tau=\sigma}^\mathfrak{M}\in U}$;
\item if $R\in L$ is an $n$-ary relation symbol, then
\[
R^{\mathfrak{M}/U}=\bp{\ap{\qp{\tau_1}_U,\dots,\qp{\tau_n}_U}\in{^n(M/U)} : \Qp{R(\tau_1,\dots,\tau_n)}^\mathfrak{M}\in U};
\]
\item if $c\in L$ is a constant symbol, then $c^{\mathfrak{M}/U}=\qp{c^\mathfrak{M}}_U$.
\end{itemize}
\end{definition}

When $\mathfrak{M}$ is full, the semantics of the quotient $\mathfrak{M}/U$ can be computed by means of the following theorem, generalizing the classic theorem of \L o\'s. For further details, see Viale \cite[Theorem 6.3.7]{viale:method}.

\begin{theorem}\label{theorem:los} Let $\bool{B}$ be a Boolean algebra and $\mathfrak{M}$ a full $\bool{B}$-valued structure for $L$. For every ultrafilter $U$ on $\bool{B}$ and every $L$-formula $\varphi(x_1,\dots,x_n)$ with parameters $\tau_1,\dots,\tau_n\in M$, we have
\[
\mathfrak{M}/U\models\varphi\bigl(\qp{\tau_1}_U,\dots,\qp{\tau_n}_U\bigr)\iff\Qp{\varphi(\tau_1,\dots,\tau_n)}^\mathfrak{M}\in U.
\]
\end{theorem}

We are interested in obtaining quotients with some degree of saturation. Let $\delta$ be a cardinal; recall that an $L$-structure $\mathfrak{M}$ is \emph{${<}\delta$-saturated} if for every $A\in\qp{M}^{<\delta}$, all types with parameters in $A$ are realized in $\mathfrak{M}$.

\begin{theorem}[Morley and Vaught \cite{mv:hum}]\label{theorem:mv} Let $\delta$ be a cardinal. If $\mathfrak{N}$ is a ${<}\delta$-saturated $L$-structure, then for every model $\mathfrak{M}\models\Th(\mathfrak{N})$ such that $\vp{M}\le\delta$ there exists an elementary embedding $j\colon\mathfrak{M}\to\mathfrak{N}$.
\end{theorem}

Recall that a sentence is \emph{universal} if it is of the form $\forall x_1\dots\forall x_n\varphi(x_1,\dots,x_n)$ where $\varphi(x_1,\dots,x_n)$ is quantifier free. Given an $L$-structure $\mathfrak{N}$, let $\Th_\forall(\mathfrak{N})$ be the set of universal $L$-sentences which are true in $\mathfrak{N}$.

\begin{theorem}[{Keisler \cite[Theorem 2.1]{keisler:preservation}}]\label{theorem:keisler} Let $\delta$ be a cardinal. If $\mathfrak{N}$ is a ${<}\delta$-saturated $L$-structure, then for every model $\mathfrak{M}\models\Th_\forall(\mathfrak{N})$ such that $\vp{M}\le\delta$ there exists an embedding $e\colon\mathfrak{M}\to\mathfrak{N}$.
\end{theorem}

We introduce goodness, a combinatorial property of ultrafilters designed to produce saturated quotients.

\begin{definition} Let $\lambda$ be a cardinal, $\bool{B}$ a Boolean algebra, and $f\colon\qp{\lambda}^{<\aleph_0}\to\bool{B}$.
\begin{itemize}
\item $f$ is \emph{monotonic} if for all $S,T\in\qp{\lambda}^{<\aleph_0}$, $S\subseteq T$ implies $f(T)\le f(S)$.
\item $f$ is \emph{multiplicative} if for all $S,T\in\qp{\lambda}^{<\aleph_0}$, $f(S\cup T)=f(S)\wedge f(T)$.
\end{itemize}
\end{definition}

If $f$ is a multiplicative function and $S\subseteq T$, then $f(T)=f(S\cup T)=f(S)\wedge f(T)\le f(S)$, whence $f$ is monotonic.

\begin{definition}[{Mansfield \cite[Definition 4.2]{MANSFIELD}}]\label{definition:mansfieldgood} Let $\delta$ be a cardinal. An ultrafilter $U$ on a Boolean algebra $\bool{B}$ is \emph{${<}\delta$-good} if for every $\lambda<\delta$ and every monotonic function $f\colon\qp{\lambda}^{<\aleph_0}\to U$, there exists a multiplicative function $g\colon\qp{\lambda}^{<\aleph_0}\to U$ with the property that $g(S)\le f(S)$ for all $S\in\qp{\lambda}^{<\aleph_0}$.
\end{definition}

For further information on good ultrafilters and their model-theoretic significance, we refer the reader to the previous work of the first author \cite[Section 2]{parente-kobu}.

\begin{theorem}\label{theorem:mansfield} Let $\delta$ be an uncountable cardinal, $\bool{B}$ a ${<}\delta$-complete Boolean algebra, and $L$ a signature with $\vp{L}<\delta$. Suppose $\mathfrak{M}$ is a full $\bool{B}$-valued structure for $L$ which satisfies the ${<}\delta$-mixing property, and $U$ is a countably incomplete ${<}\delta$-good ultrafilter on $\bool{B}$; then the quotient $\mathfrak{M}/U$ is ${<}\delta$-saturated.
\end{theorem}
\begin{proof} Mansfield \cite[Theorem 4.1]{MANSFIELD} established this result for the special case of \emph{Boolean ultrapowers} $M^{(\bool{B})}/U$. However, by inspecting his proof one sees that fullness and ${<}\delta$-mixing are the only properties of the $\bool{B}$-valued structure actually used to obtain ${<}\delta$-saturation of the quotient. Thus, the very same proof applies to our case as well.
\end{proof}

The rest of this section is dedicated to the existence of good ultrafilters. Balcar and Franek \cite{balcarfranek} showed that if $\bool{B}$ is an infinite complete Boolean algebra such that $\sat(\bool{B}\rest b)=\sat(\bool{B})$ for all $b\in\bool{B}^+$, then there exist $2^{\vp{\bool{B}}}$ ultrafilters on $\bool{B}$ which are countably incomplete and ${<}\sat(\bool{B})$-good. Although this is sufficient for many purposes, it relies on the completeness of the Boolean algebra, whereas in Section~\ref{section:cinque} we shall need to work with the stationary tower, which is not complete.

We overcome this by giving a new existence proof closer to the original proof of Keisler \cite[Theorem 4.4]{keisler:good}.

\begin{definition}[{Mansfield \cite[Definition 4.3]{MANSFIELD}}] Let $\delta$ be a cardinal; a Boolean algebra $\bool{B}$ is \emph{${<}\delta$-disjointable} if for every $\lambda<\delta$ and every function $f\colon\lambda\to\bool{B}^+$, there exists a function $g\colon\lambda\to\bool{B}^+$ such that:
\begin{itemize}
\item for all $\alpha<\lambda$, $g(\alpha)\le f(\alpha)$;
\item for all $\beta<\alpha<\lambda$, $g(\alpha)\wedge g(\beta)=\bool{0}$.
\end{itemize}
\end{definition}

The following sufficient condition will be easier to verify for successor cardinals.

\begin{theorem}[{Balcar and Vojtáš \cite[Theorem 1.4]{balcarvojtas}}]\label{theorem:bv} Let $\lambda$ be an infinite cardinal and $\bool{B}$ a Boolean algebra. If $\lambda^+<\sat(\bool{B}\rest b)$ for all $b\in\bool{B}^+$, then $\bool{B}$ is ${<}\lambda^+$-disjointable.
\end{theorem}

We now look at two consequences of ${<}\delta$-disjointability.

\begin{lemma}[Balcar and Vojtáš \cite{balcarvojtas}]\label{lemma:bv} Let $\delta>2$ be a cardinal. If a Boolean algebra $\bool{B}$ is ${<}\delta$-disjointable, then $\bool{B}$ is atomless and for all $b\in\bool{B}^+$ $\delta\le\dens(\bool{B}\rest b)$.
\end{lemma}

\begin{proposition}\label{proposition:disjcard} Let $\delta$ be a cardinal and $\bool{B}$ a Boolean algebra. If $\bool{B}$ is ${<}\delta$-disjointable and ${<}\delta$-complete, then $\delta^{<\delta}\le\vp{\bool{B}}$.
\end{proposition}
\begin{proof} If $\delta\le 2$ there is nothing to prove, so let us assume $\delta>2$. Let $\lambda<\delta$ be arbitrary; since $\bool{B}$ is ${<}\delta$-disjointable, in particular there exists an antichain $A$ in $\bool{B}$ such that $\vp{A}=\lambda$. By Lemma \ref{lemma:bv}, for all $a\in A$ we have $\delta\le\dens(\bool{B}\rest a)$.

Let $\prod_{a\in A}\bool{B}\rest a$ denote the product algebra of the $\bool{B}\rest a$ for $a\in A$. Clearly, the function defined as
\[
\begin{split}
\prod_{a\in A}\bool{B}\rest a &\longrightarrow \bool{B} \\
\ap{b_a : a\in A} &\longmapsto \bigvee_{a\in A}b_a
\end{split}
\]
is injective. Hence,
\[
\vp{\bool{B}}\ge\prod_{a\in A}\vp{\bool{B}\rest a}\ge\prod_{a\in A}\dens(\bool{B}\rest a)\ge\prod_{a\in A}\delta=\delta^\lambda,
\]
as desired.
\end{proof}

The next lemma will constitute the base step of the recursive construction of Theorem \ref{theorem:egood}. Our proof is not different from the proof of Keisler \cite[Lemma 4B]{keisler:good}.

\begin{lemma}\label{lemma:4b} Let $\delta$ be an infinite cardinal; suppose $\bool{B}$ is a ${<}\delta$-disjointable ${<}\delta$-complete Boolean algebra. For every $E\in\qp{\bool{B}}^{<\delta}$ with the finite intersection property, there exists $Y\in\qp{\bool{B}}^{\aleph_0}$ such that $E\cup Y$ has the finite intersection property and $\bigwedge Y=\bool{0}$.
\end{lemma}
\begin{proof} Without loss of generality, we may assume that $E$ is also closed under finite meets. Let $f\colon\lambda\to E$ be a surjective function for some $\lambda<\delta$. Since $\bool{B}$ is ${<}\delta$-disjointable, there exists a function $g\colon\lambda\to\bool{B}^+$ such that $g(\alpha)\le f(\alpha)$ and $g(\alpha)\wedge g(\beta)=\bool{0}$ for all $\beta<\alpha<\lambda$. By Lemma \ref{lemma:bv}, $\bool{B}$ is atomless and therefore for every $\alpha<\lambda$ we can choose a sequence $\ap{x_{\alpha,n} : n<\omega}$ such that:
\begin{enumerate}
\item\label{k1} $x_{\alpha,0}=g(\alpha)$;
\item\label{k2} for all $n<\omega$, $x_{\alpha,n+1}<x_{\alpha,n}$;
\item\label{k3} $\bigwedge_{n<\omega}x_{\alpha,n}=\bool{0}$.
\end{enumerate}
For every $n<\omega$ let us define $y_n=\bigvee_{\alpha<\lambda}x_{\alpha,n}$ and prove that $Y=\bp{y_n : n<\omega}$ has the desired properties.

Firstly, to prove that $E\cup Y$ has the finite intersection property, we observe that $Y$ closed under finite meets, in fact $y_{n+1}\le y_n$ for all $n<\omega$. Then for every $\alpha<\lambda$ and $n<\omega$ we have
\[
f(\alpha)\wedge y_n\ge g(\alpha)\wedge y_n=x_{\alpha,n}>\bool{0},
\]
as desired.

Secondly, we show that $\bigwedge Y=\bool{0}$. If not, then there must be some $b\in\bool{B}^+$ such that $b\le y_n$ for each $n<\omega$. In particular, from point \eqref{k1} we get $b\le y_0=\bigvee_{\alpha<\lambda}g(\alpha)$, hence there exists $\beta<\lambda$ such that $g(\beta)\wedge b>\bool{0}$. Using point \eqref{k3}, we can find some $m<\omega$ such that $g(\beta)\wedge b\wedge\neg x_{\beta,m}>\bool{0}$. Thus we obtain
\[
\bool{0}<g(\beta)\wedge b\wedge\neg x_{\beta,m}\le g(\beta)\wedge y_m\wedge\neg x_{\beta,m}=x_{\beta,m}\wedge\neg x_{\beta,m}=\bool{0},
\]
which is a contradiction.
\end{proof}

The following lemma will take care of the successor steps of the recursive construction.

\begin{lemma}[{Mansfield \cite[Theorem 4.2]{MANSFIELD}}]\label{lemma:4c} Let $\delta$ be an infinite cardinal; suppose $\bool{B}$ is a ${<}\delta$-disjointable ${<}\delta$-complete Boolean algebra. For every $E\in\qp{\bool{B}}^{<\delta}$ with the finite intersection property, every $\lambda<\delta$, and every monotonic function $f\colon\qp{\lambda}^{<\aleph_0}\to E$, there exists $E'\in\qp{\bool{B}}^{<\delta}$ with the finite intersection property such that $E\subseteq E'$, and a multiplicative function $g\colon\qp{\lambda}^{<\aleph_0}\to E'$ such that $g(S)\le f(S)$ for all $S\in\qp{\lambda}^{<\aleph_0}$.
\end{lemma}

Finally, we apply Proposition \ref{proposition:disjcard}, Lemma \ref{lemma:4b}, and Lemma \ref{lemma:4c} to establish an existence result for good ultrafilters.

\begin{theorem}\label{theorem:egood} Let $\delta$ be an infinite cardinal; suppose $\bool{B}$ is a ${<}\delta$-disjointable ${<}\delta$-complete Boolean algebra of cardinality $\delta$. For every $E\in\qp{\bool{B}}^{<\delta}$ with the finite intersection property, there exists a countably incomplete ${<}\delta$-good ultrafilter $U$ on $\bool{B}$ such that $E\subset U$.
\end{theorem}
\begin{proof}
If $E\in\qp{\bool{B}}^{<\delta}$ has the finite intersection property, by Lemma \ref{lemma:4b} there exists $Y\in\qp{\bool{B}}^{\aleph_0}$ such that $E\cup Y$ has the finite intersection property and $\bigwedge Y=\bool{0}$. Now, if $\delta=\aleph_0$ there is nothing else to prove: any ultrafilter including $E\cup Y$ will be countably incomplete and, trivially, ${<}\aleph_0$-good. Therefore, for the rest of the proof we shall assume that $\delta$ is uncountable.

Since $\vp{\bool{B}}=\delta$, let us enumerate $\bool{B}=\bp{b_\alpha : \alpha<\delta}$. By Proposition \ref{proposition:disjcard} we know that $\delta^{<\delta}=\delta$, therefore the set
\[
\bp{f\colon\qp{\lambda}^{<\aleph_0}\to\bool{B} : \lambda<\delta\text{ and }f\text{ is monotonic}}
\]
can be enumerated as $\bp{f_\alpha : \alpha<\delta}$ in such a way that each function is listed cofinally many times.

We recursively construct a sequence $\ap{D_\alpha : \alpha<\delta}$ such that, at each stage $\alpha<\delta$:
\begin{enumerate}
\item\label{gooduno} $D_\alpha\in\qp{\bool{B}^+}^{<\delta}$ has the finite intersection property;
\item\label{gooddue} if $\beta<\alpha$ then $D_\beta\subseteq D_\alpha$;
\item\label{goodtre} either $b_\alpha\in D_{\alpha+1}$ or $\neg b_\alpha\in D_{\alpha+1}$;
\item\label{goodquattro} if $f_\alpha\colon\qp{\lambda}^{<\aleph_0}\to D_\alpha$ where $\lambda<\delta$, then there exists a multiplicative function $g\colon\qp{\lambda}^{<\aleph_0}\to D_{\alpha+1}$ such that $g(S)\le f_\alpha(S)$ for all $S\in\qp{\lambda}^{<\aleph_0}$.
\end{enumerate}
We begin the recursive construction by letting $D_0=E\cup Y$. Clearly, $\vp{D_0}\le\vp{E}+\aleph_0<\delta$, as we are assuming $\delta$ to be uncountable.

For the inductive step, assume $D_\alpha$ has already been constructed. We first define $D_\alpha'$ as follows. In case the image of $f_\alpha$ is not included in $D_\alpha$, then let $D_\alpha'=D_\alpha$. On the other hand, suppose $f_\alpha\colon\qp{\lambda}^{<\aleph_0}\to D_\alpha$ for some $\lambda<\delta$. Then, by Lemma \ref{lemma:4c} there exists $D_\alpha'\in\qp{\bool{B}}^{<\delta}$ with the finite intersection property such that $D_\alpha\subseteq D_\alpha'$, and a multiplicative function $g\colon\qp{\lambda}^{<\aleph_0}\to D_\alpha'$ such that $g(S)\le f_\alpha(S)$ for all $S\in\qp{\lambda}^{<\aleph_0}$. Note that at least one of $D_\alpha'\cup\bp{b_\alpha}$ and $D_\alpha'\cup\bp{\neg b_\alpha}$ has the finite intersection property, hence we may define $D_{\alpha+1}=D_\alpha'\cup\bp{b_\alpha}$ in the former case, or $D_{\alpha+1}=D_\alpha'\cup\bp{\neg b_\alpha}$ in the latter case.

When $\gamma<\delta$ is a limit ordinal, we take $D_\gamma=\bigcup_{\alpha<\gamma}D_\alpha$. Finally, let $U=\bigcup_{\alpha<\delta} D_\alpha$. By condition \eqref{goodtre}, the filter $U$ is indeed an ultrafilter. Furthermore, by our choice of $D_0$, $U$ is countably incomplete and includes $E$. To see that $U$ is ${<}\delta$-good, let $\lambda<\delta$ and let $f\colon\qp{\lambda}^{<\aleph_0}\to U$ be a monotonic function. Since $\delta$ is a regular cardinal, there exists $\beta<\delta$ such that $f\colon\qp{\lambda}^{<\aleph_0}\to D_\beta$. Now choose some $\alpha>\beta$ such that $f=f_\alpha$; by \eqref{goodquattro} there exists a multiplicative function $g\colon\qp{\lambda}^{<\aleph_0}\to D_{\alpha+1}$ such that $g(S)\le f(S)$ for all $S\in\qp{\lambda}^{<\aleph_0}$, as desired.
\end{proof}

\section{The collapsing algebra}\label{section:quattro}

The purpose of this section is to show that, for an infinite regular cardinal $\kappa$ and an inaccessible $\delta>\kappa$, there exist a $\delta$-preserving complete Boolean algebra $\bool{B}$ and ultrafilters $U$ on $\bool{B}$, such that every model of the universal theory of $H_{\kappa^+}$ of cardinality $\le\delta$ is embeddable into $H_{\check{\delta}}^\bool{B}/U$.

Let $\delta$ be an inaccessible cardinal. For every infinite regular cardinal $\kappa<\delta$, consider the notion of forcing
\[
P^\kappa_{<\delta}=\bp{p : p\text{ is a function},\, \vp{p}<\kappa,\, \dom(p)\subset\delta\times\kappa,\, \forall\ap{\alpha,\xi}\in\dom(p)(p(\alpha,\xi)\in\alpha\cup\bp{0})}
\]
ordered by $q\le p\iff p\subseteq q$; then define the \emph{collapsing algebra} $\Coll(\kappa,{<}\delta)=\RO(P^\kappa_{<\delta})$. It is well known that $\Coll(\kappa,{<}\delta)$ has cardinality $\delta$ and includes $P^\kappa_{<\delta}$ as a ${<}\kappa$-closed dense subset.

\begin{proposition}\label{proposition:colldisj} $\Coll(\kappa,{<}\delta)$ is ${<}\delta$-disjointable.
\end{proposition}
\begin{proof} We first claim that for all $p\in P^\kappa_{<\delta}$ and $\lambda<\delta$ there exists an antichain $\bp{q_\beta : \beta<\lambda}$ in $P^\kappa_{<\delta}$ such that $q_\beta\le p$ for each $\beta<\lambda$. Given a condition $p$, it is possible to choose $\alpha$ such that $\lambda\le\alpha<\delta$ and $\ap{\alpha,0}\notin\dom(p)$. Then, for each $\beta<\lambda$ take $q_\beta=p\cup\bp{\ap{\ap{\alpha,0},\beta}}$.

By density, this implies that for all $b\in\Coll(\kappa,{<}\delta)^+$ we have $\delta\le\sat(\Coll(\kappa,{<}\delta)\rest b)$. Hence, by Theorem \ref{theorem:bv} we conclude that $\Coll(\kappa,{<}\delta)$ is ${<}\delta$-disjointable.
\end{proof}

\begin{remark} To prove Proposition \ref{proposition:colldisj} it is not necessary to use Theorem \ref{theorem:bv}. In fact, the ${<}\delta$-disjointability of $\Coll(\kappa,{<}\delta)$ can also be established directly by a density argument about $P^\kappa_{<\delta}$.
\end{remark}

\begin{corollary}\label{corollary:coll} For every $E\in\qp{\Coll(\kappa,{<}\delta)}^{<\delta}$ with the finite intersection property, there exists a countably incomplete ${<}\delta$-good ultrafilter $U$ on $\Coll(\kappa,{<}\delta)$ such that $E\subset U$.
\end{corollary}
\begin{proof} By Theorem \ref{theorem:egood} and Proposition \ref{proposition:colldisj}.
\end{proof}

Note that, since $\Coll(\kappa,{<}\delta)$ is $\delta$-preserving, the Boolean-valued structure $H_{\check{\delta}}^{\Coll(\kappa,{<}\delta)}$ represents $H_\delta$ of the forcing extension, in the sense of Proposition \ref{proposition:rep}.

\begin{theorem}\label{theorem:mainuno} Let $\delta$ be an inaccessible cardinal and $\kappa<\delta$ be an infinite regular cardinal. Every $E\in\qp{\Coll(\kappa,{<}\delta)}^{<\delta}$ with the finite intersection property can be extended to an ultrafilter $U$ on $\Coll(\kappa,{<}\delta)$ such that:
\begin{itemize}
\item the quotient $\ap{H_{\check{\delta}}^{\Coll(\kappa,{<}\delta)}/U,\in_{\Delta_0}}$ is a model of $\Th_\forall(H_{\kappa^+},\in_{\Delta_0})$;
\item for every model $\mathfrak{M}\models\Th_\forall(H_{\kappa^+},\in_{\Delta_0})$ such that $\vp{M}\le\delta$, there exists an embedding $e\colon\mathfrak{M}\to\ap{H_{\check{\delta}}^{\Coll(\kappa,{<}\delta)}/U,\in_{\Delta_0}}$ of $\in_{\Delta_0}$-structures.
\end{itemize}
\end{theorem}
\begin{proof} Applying Corollary \ref{corollary:coll}, let $U$ be a countably incomplete ${<}\delta$-good ultrafilter on $\Coll(\kappa,{<}\delta)$ such that $E\subset U$. From Proposition \ref{proposition:mixing} we know that $H_{\check{\delta}}^{\Coll(\kappa,{<}\delta)}$ satisfies the mixing property, hence it is full by Proposition \ref{proposition:mixingfull}.

To prove the first point, let $\psi$ be a universal $\in_{\Delta_0}$-sentence which is true in $H_{\kappa^+}$. Without loss of generality, $\psi$ is of the form $\forall x\varphi(x)$ where $\varphi(x)$ is quantifier free. Since
\[
\ap{H_{\kappa^+},\in_{\Delta_0}}\models\forall x\varphi(x),
\]
Lemma \ref{lemma:coll} gives $\Qp{\forall x\varphi(x)}^{V^{\Coll(\kappa,{<}\delta)}}=\bool{1}$, which easily implies $\Qp{\forall x\varphi(x)}^{H_{\check{\delta}}^{\Coll(\kappa,{<}\delta)}}=\bool{1}$. It follows that
\[
\ap{H_{\check{\delta}}^{\Coll(\kappa,{<}\delta)}/U,\in_{\Delta_0}}\models\forall x\varphi(x),
\]
which is what we had to prove.

For the second point, let $\mathfrak{M}$ be any model of $\Th_\forall(H_{\kappa^+},\in_{\Delta_0})$ such that $\vp{M}\le\delta$. Remark \ref{remark:abs} implies that $\mathfrak{M}$ is also a model of $\Th_\forall(H_{\check{\delta}}^{\Coll(\kappa,{<}\delta)}/U,\in_{\Delta_0})$. Then Theorem \ref{theorem:keisler} and Theorem \ref{theorem:mansfield} give the desired embedding.
\end{proof}

The main drawbacks of Theorem \ref{theorem:mainuno} are, first, the requirement of $\kappa$ being regular and, second, the signature being restricted to $\Delta_0$-definable relations. In the next section, we overcome both limitations by using the stationary tower to establish a stronger universality result.

\section{The stationary tower forcing}\label{section:cinque}

In parallel with the previous section, we provide a second universality result for models of the form $\Ult(H_{\kappa^+},\tow{I})/U$, where $\tow{I}$ is a tower of non-stationary ideals. We employ the following definition of stationarity, which is due to Shelah; see also Woodin \cite[Definition 3]{WOO88} and the monograph of Larson \cite{LARSON}.

\begin{definition} A non-empty set $a$ is \emph{stationary} if for every function $f\colon{^{<\omega}{(\U a)}}\to\U a$ there exists $Z\in a$ such that $f[^{<\omega}Z]\subseteq Z$.

A set $a$ is \emph{stationary in} $b$ if $a$ is stationary, $a\subseteq b$, and $\U a=\U b$.
\end{definition}

We clarify the relation between stationarity and Jech's notion of closed unbounded sets.

\begin{definition}[{Jech \cite[Definition 3.1]{jech.combinatorial_problems}}] Let $\kappa$ be an infinite cardinal and $X$ a set such that $\kappa^+\subseteq X$. A subset $C\subseteq\qp{X}^{\le\kappa}$ is
\begin{itemize}
\item \emph{closed} if it is closed under unions of chains of length $\le\kappa$,
\item \emph{unbounded} if for every $Z\in\qp{X}^{\le\kappa}$ there exists $Y\in C$ such that $Z\subseteq Y$.
\end{itemize}
\end{definition}

A key example is the set
\[
\mathcal{P}_{\kappa^+}(X)=\bp{Z\in\qp{X}^{\le\kappa} : Z\cap\kappa^+\in\kappa^+},
\]
which is closed and unbounded.

It follows from the definitions that if $a\subseteq\qp{X}^{\le\kappa}$ has non-empty intersection with every closed unbounded set, then $a$ is stationary in $\qp{X}^{\le\kappa}$. The converse is not true in general, for instance the set of countable subsets of $\omega_2$ is stationary, but it is disjoint from the closed unbounded set $\qp{\omega_2}^{\aleph_1}$. However, we have the following result.

\begin{proposition}[Kueker \cite{kueker}] For every closed unbounded $C\subseteq\qp{X}^{\le\kappa}$ there exists a function $f\colon {^{<\omega}X}\to X$ such that $\bp{Z\in\mathcal{P}_{\kappa^+}(X) : f[^{<\omega}Z]\subseteq Z}\subseteq C$.
\end{proposition}

As a consequence, a subset $a\subseteq\mathcal{P}_{\kappa^+}(X)$ has non-empty intersection with every closed unbounded set if and only if $a$ is stationary in $\mathcal{P}_{\kappa^+}(X)$. Let us also recall a result on splitting stationary sets, which will be useful for our Proposition \ref{proposition:towdisj}.

\begin{theorem}[{Matsubara \cite[Theorem 2]{matsubara}}]\label{theorem:matsubara} Every set which is stationary in $\mathcal{P}_{\kappa^+}(X)$ can be partitioned into $\vp{X}$ disjoint stationary sets.
\end{theorem}

If $\kappa$ is any infinite cardinal and $\kappa^+\subseteq X$, we define the \emph{non-stationary ideal}
\[
\mathrm{NS}\rest\mathcal{P}_{\kappa^+}(X)=\bp{a\subseteq\mathcal{P}(X) : a\cap\mathcal{P}_{\kappa^+}(X)\text{ is not stationary in }\mathcal{P}_{\kappa^+}(X)},
\]
which is a normal fine ideal over $\mathcal{P}(X)$ (see \cite[Section 3]{FOREMAN}). Furthermore, a well-known result of Menas \cite[Corollary 1.9]{menas} gives that for all $\kappa^+\subseteq X\subseteq Y$
\begin{equation}\label{eq:menas}
\mathrm{NS}\rest\mathcal{P}_{\kappa^+}(X)=\bp{a\subseteq\mathcal{P}(X) : \bp{Z\subseteq Y : Z\cap X\in a}\in\mathrm{NS}\rest\mathcal{P}_{\kappa^+}(Y)}.
\end{equation}
As a result, whenever $\delta>\kappa$ is inaccessible we have a tower
\[
\tow{I}^\kappa_{<\delta}=\ap{\mathrm{NS}\rest\mathcal{P}_{\kappa^+}(X) : \kappa^+\subseteq X\in V_\delta}
\]
of height $\delta$.

The corresponding Boolean algebra $\bool{B}(\tow{I}^\kappa_{<\delta})$ is called the \emph{stationary tower} of height $\delta$ and critical point $\kappa^+$. To view concretely the stationary tower as a notion of forcing, let us define
\[
Q^\kappa_{<\delta}=\bp{a\in V_\delta : \kappa^+\subseteq\U a\text{ and }a\text{ is stationary in }\mathcal{P}_{\kappa^+}(\U a)}
\]
and for $a,b\in Q^\kappa_{<\delta}$ define
\[
b\le a\iff\U a\subseteq\U b\text{ and for all }Z\in b,\ Z\cap\U a\in a.
\]
Then the map $a\mapsto\qp{a}_{\tow{I}^\kappa_{<\delta}}$ gives a dense embedding of $Q^\kappa_{<\delta}$ into ${\bool{B}(\tow{I}^\kappa_{<\delta})}^+$. In particular, since $\tow{I}^\kappa_{<\delta}$ concentrates on sets of cardinality less than or equal to $\kappa$, for every $\qp{a}_{\tow{I}^\kappa_{<\delta}}\in\bool{B}(\tow{I}^\kappa_{<\delta})$ we have  $\qp{a}_{\tow{I}^\kappa_{<\delta}}=\qp{\bp{Z\in a : \vp{Z}\le\kappa}}_{\tow{I}^\kappa_{<\delta}}$.

\begin{proposition}\label{proposition:towdisj} Let $\delta$ be inaccessible; for every infinite cardinal $\kappa<\delta$ the Boolean algebra $\bool{B}(\tow{I}^\kappa_{<\delta})$ is ${<}\delta$-disjointable.
\end{proposition}
\begin{proof} We first claim that for every $a\in Q^\kappa_{<\delta}$ and every $\lambda<\delta$, if $\U a\subseteq V_\lambda$ then there exists an antichain $\bp{b_\beta : \beta<\lambda}$ in $Q^\kappa_{<\delta}$ such that $b_\beta\le a$ for each $\beta<\lambda$. Given a condition $a$, by \eqref{eq:menas} the set
\[
\bp{Z\in\mathcal{P}_{\kappa^+}(V_\lambda) : Z\cap\U a\in a}
\]
is stationary in $\mathcal{P}_{\kappa^+}(V_\lambda)$, hence by Theorem \ref{theorem:matsubara} we can partition it into $\lambda$ disjoint stationary sets $b_\beta$ for $\beta<\lambda$. Then it is clear that $\bp{b_\beta : \beta<\lambda}$ is the desired antichain in $Q^\kappa_{<\delta}$.

By density, this implies that for all $b\in{\bool{B}(\tow{I}^\kappa_{<\delta})}^+$ we have $\delta\le\sat(\bool{B}(\tow{I}^\kappa_{<\delta})\rest b)$. Therefore, Theorem \ref{theorem:bv} gives that $\bool{B}(\tow{I}^\kappa_{<\delta})$ is ${<}\delta$-disjointable.
\end{proof}

\begin{corollary}\label{corollary:tow} For every $E\in\qp{\bool{B}(\tow{I}^\kappa_{<\delta})}^{<\delta}$ with the finite intersection property, there exists a countably incomplete ${<}\delta$-good ultrafilter $U$ on $\bool{B}(\tow{I}^\kappa_{<\delta})$ such that $E\subset U$.
\end{corollary}
\begin{proof} By Lemma \ref{lemma:tow}, Theorem \ref{theorem:egood}, and Proposition \ref{proposition:towdisj}.
\end{proof}

Next, we turn to the relevant Boolean-valued structure. As opposed to the case of Section \ref{section:quattro}, here $H_{\check{\delta}}^{\bool{B}(\tow{I}^\kappa_{<\delta})}$ would not naturally be a $\bool{B}(\tow{I}^\kappa_{<\delta})$-valued structure, simply because the Boolean algebra fails to be complete. One approach is pursued in the work of Marinov \cite{MARINOV}: quotient $\bool{B}(\tow{I}^\kappa_{<\delta})$ by a carefully chosen ideal $I$ to obtain a complete quotient algebra and then work with $H_{\check{\delta}}^{\bool{B}(\tow{I}^\kappa_{<\delta})/I}$. Here we choose a different route and argue that $\Ult(H_{\kappa^+},\tow{I}^\kappa_{<\delta})$ is a suitable $\bool{B}(\tow{I}^\kappa_{<\delta})$-valued structure, which then we quotient by a ${<}\delta$-good ultrafilter to obtain our main universality result.

\begin{theorem}\label{theorem:maindue} Let $\delta$ be an inaccessible cardinal and $\kappa<\delta$ be an infinite cardinal. Every $E\in\qp{\bool{B}(\tow{I}^\kappa_{<\delta})}^{<\delta}$ with the finite intersection property can be extended to an ultrafilter $U$ on $\bool{B}(\tow{I}^\kappa_{<\delta})$ such that for every $\mathcal{A}\subseteq\bigcup_{n<\omega}\mathcal{P}(^n(H_{\kappa^+}))$:
\begin{itemize}
\item the quotient $\ap{\Ult(H_{\kappa^+},\tow{I}^\kappa_{<\delta})/U,\in_\mathcal{A}}$ is a model of $\Th(H_{\kappa^+},\in_\mathcal{A})$;
\item for every model $\mathfrak{M}\models\Th_\forall(H_{\kappa^+},\in_\mathcal{A})$ such that $\vp{M}\le\delta$, there exists an embedding $e\colon\mathfrak{M}\to\ap{\Ult(H_{\kappa^+},\tow{I}^\kappa_{<\delta})/U,\in_\mathcal{A}}$ of $\in_\mathcal{A}$-structures;
\item for every model $\mathfrak{M}\models\Th(H_{\kappa^+},\in_\mathcal{A})$ such that $\vp{M}\le\delta$, there exists an elementary embedding $j\colon\mathfrak{M}\to\ap{\Ult(H_{\kappa^+},\tow{I}^\kappa_{<\delta})/U,\in_\mathcal{A}}$ of $\in_\mathcal{A}$-structures.
\end{itemize}
\end{theorem}
\begin{proof} Applying Corollary \ref{corollary:tow}, let $U$ be a countably incomplete ${<}\delta$-good ultrafilter on $\bool{B}(\tow{I}^\kappa_{<\delta})$ such that $E\subset U$. Proposition \ref{proposition:full} ensures that $\Ult(H_{\kappa^+},\tow{I}^\kappa_{<\delta})$ is full and satisfies the ${<}\delta$-mixing property.

The first point follows directly from Corollary \ref{corollary:abs}. For the second point, let $\mathfrak{M}$ be a model of $\Th_\forall(H_{\kappa^+},\in_\mathcal{A})$ such that $\vp{M}\le\delta$. Corollary \ref{corollary:abs} again implies that $\mathfrak{M}\models\Th_\forall(\Ult(H_{\kappa^+},\tow{I}^\kappa_{<\delta})/U,\in_\mathcal{A})$. Then Theorem \ref{theorem:keisler} and Theorem \ref{theorem:mansfield} give the desired embedding. The third point is completely analogous to the second, except that we apply Theorem \ref{theorem:mv} instead of Theorem \ref{theorem:keisler} to obtain an elementary embedding.
\end{proof}

For the universality result of the previous section, we considered Boolean-valued structures $H_{\check{\delta}}^{\Coll(\kappa,{<}\delta)}$ which, by Proposition \ref{proposition:rep}, provide a canonical representation for $H_\delta$ when forcing with $\Coll(\kappa,{<}\delta)$. In the remainder of this section, we aim to show that a similar fact holds for the structures we consider in Theorem \ref{theorem:maindue}. More precisely, we shall show that under large cardinal assumptions, quotienting $\Ult(H_{\kappa^+},\tow{I}^\kappa_{<\delta})$ by a $\bool{B}(\tow{I}^\kappa_{<\delta})$-generic filter $G$ gives rise to a well-founded structure whose transitive collapse is $H_\delta$ as computed in $V[G]$.

\begin{definition} A cardinal $\delta$ is \emph{Woodin} if for all $f\colon\delta\to\delta$ there exists $\alpha<\delta$ with $f[\alpha]\subseteq\alpha$ and a nontrivial elementary embedding $j\colon V\to M$ such that $\crit(j)=\alpha$ and $V_{j(f)(\alpha)}\subset M$.
\end{definition}

If $\delta$ is a Woodin cardinal, then $\delta$ is Mahlo (see Kanamori \cite[Exercise 26.10]{kanamori:higher_infinite}) and therefore inaccessible.

\begin{theorem}[Woodin]\label{theorem:woodin} Suppose $\delta$ is a Woodin cardinal. Then, for every infinite $\kappa<\delta$ the Boolean algebra $\bool{B}(\tow{I}^\kappa_{<\delta})$ is ${<}\delta$-presaturated. Furthermore, if $G$ is $\bool{B}(\tow{I}^\kappa_{<\delta})$-generic over $V$ then $\Ult(H_{\kappa^+},\tow{I}^\kappa_{<\delta})/G$ is well founded.
\end{theorem}
\begin{proof} Combine \cite[Proposition 9.2]{FOREMAN} and \cite[Corollary 9.24]{FOREMAN}.
\end{proof}

Accordingly, if $\delta$ is Woodin and $G$ is $\bool{B}(\tow{I}^\kappa_{<\delta})$-generic over $V$, then we identify $\Ult(H_{\kappa^+},\tow{I}^\kappa_{<\delta})/G$ with its transitive collapse and let
\[
\begin{split}
j_G\colon H_{\kappa^+}&\longrightarrow\Ult(H_{\kappa^+},\tow{I}^\kappa_{<\delta})/G \\
a &\longmapsto\qp{c_a}_G
\end{split}
\]
be the generic elementary embedding given by Corollary \ref{corollary:abs}. Since $\tow{I}^\kappa_{<\delta}$ concentrates on those $Z$ such that $Z\cap\kappa^+\in\kappa^+$, using \cite[Proposition 4.56]{FOREMAN} we see that $j_G(\alpha)=\alpha$ for all $\alpha<\kappa^+$.

\begin{proposition} Let $\delta$ be a Woodin cardinal and $\kappa<\delta$ be an infinite cardinal. If $G$ is $\bool{B}(\tow{I}^\kappa_{<\delta})$-generic over $V$, then
\[
H_\delta^{V[G]}=\Ult(H_{\kappa^+},\tow{I}^\kappa_{<\delta})/G.
\]
\end{proposition}
\begin{proof} Theorem \ref{theorem:woodin} guarantees that $\bool{B}(\tow{I}^\kappa_{<\delta})$ is ${<}\delta$-presaturated, so we may apply Lemma \ref{lemma:av} to obtain that $\delta$ is a regular cardinal in $V[G]$ and $H_\delta^{V[G]}=\bp{\tau_G : \tau\in V^{\bool{B}(\tow{I}^\kappa_{<\delta})}\cap H_\delta}$.

To prove that $H_\delta^{V[G]}\subseteq\Ult(H_{\kappa^+},\tow{I}^\kappa_{<\delta})/G$, it is enough to define a function $f\colon V^{\bool{B}(\tow{I}^\kappa_{<\delta})}\cap H_\delta\to\Ult(H_{\kappa^+},\tow{I}^\kappa_{<\delta})$ in $V$ with the property that $\qp{f(\tau)}_G=\tau_G$. Given $\tau\in V^{\bool{B}(\tow{I}^\kappa_{<\delta})}\cap H_\delta$, enumerate it as
\[
\tau=\bp{\ap{\sigma_\alpha,\qp{a_\alpha}_{\tow{I}^\kappa_{<\delta}}} : \alpha<\lambda}
\]
for some $\lambda<\delta$ and assume inductively that each $f(\sigma_\alpha)\colon\mathcal{P}(X_\alpha)\to H_{\kappa^+}$ for $\alpha<\lambda$ has already been defined in such a way that $\qp{f(\sigma_\alpha)}_G={(\sigma_\alpha)}_G$. Let $X\in V_\delta$ be sufficiently large that $\lambda\cup X_\alpha\cup\U a_\alpha\subseteq X$ for each $\alpha<\lambda$ and define $f(\tau)\colon\mathcal{P}(X)\to H_{\kappa^+}$ as follows: given $Z\subseteq X$, if $\vp{Z}\le\kappa$ then let
\[
f(\tau)(Z)=\bp{f(\sigma_\alpha)(Z\cap X_\alpha) : \alpha\in Z\cap\lambda\text{ and }Z\cap\U a_\alpha\in a_\alpha}.
\]
Otherwise, let $f(\tau)(Z)$ be defined arbitrarily.

Keeping in mind Lemma \ref{lemma:tow}, it is easy to verify that for any $g\in\Ult(H_{\kappa^+},\tow{I}^\kappa_{<\delta})$
\[
\Qp{g\in f(\tau)}^{\Ult(H_{\kappa^+},\tow{I}^\kappa_{<\delta})}=\bigvee\bp{\Qp{g=f(\sigma_\alpha)}^{\Ult(H_{\kappa^+},\tow{I}^\kappa_{<\delta})}\wedge\qp{a_\alpha}_{\tow{I}^\kappa_{<\delta}} : \alpha<\lambda}.
\]
Using the inductive hypothesis and the genericity of $G$, we conclude that
\begin{align*}
\tau_G&=\bp{{(\sigma_\alpha)}_G : \alpha<\lambda\text{ and }\qp{a_\alpha}_{\tow{I}^\kappa_{<\delta}}\in G}\\ &=\bp{\qp{f(\sigma_\alpha)}_G : \alpha<\lambda\text{ and }\qp{a_\alpha}_{\tow{I}^\kappa_{<\delta}}\in G}\\ &=\bp{\qp{g}_G : \Qp{g\in f(\tau)}^{\Ult(H_{\kappa^+},\tow{I}^\kappa_{<\delta})}\in G}\\&=\qp{f(\tau)}_G,
\end{align*}
as we wanted to show.

For the reverse inclusion, first note that
\[
\ap{H_{\kappa^+},\in_{\Delta_0}}\models\forall x(\vp{x}\le\kappa)
\]
in $V$. By the elementarity of $j_G$, it follows that
\[
\ap{\Ult(H_{\kappa^+},\tow{I}^\kappa_{<\delta})/G,\in_{\Delta_0}}\models\forall x(\vp{x}\le j_G(\kappa))
\]
in $V[G]$. Since $j_G(\kappa)=\kappa<\delta$, we have that $\Ult(H_{\kappa^+},\tow{I}^\kappa_{<\delta})/G$ is transitive and all its elements have cardinality less than $\delta$, whence $\Ult(H_{\kappa^+},\tow{I}^\kappa_{<\delta})/G\subseteq H_\delta^{V[G]}$.
\end{proof}

\end{document}